\documentclass{amsart}
\usepackage{amsfonts}
\usepackage{amssymb,amsmath}

\newtheorem{theorem}{Theorem}
\newtheorem{acknowledgement}[theorem]{Acknowledgement}

\newtheorem{corollary}[theorem]{Corollary}

\newtheorem{lemma}[theorem]{Lemma}

\newtheorem{proposition}[theorem]{Proposition}

\begin{document}
\title{On unit stable range matrices}
\author{Grigore C\u{a}lug\u{a}reanu}
\address{Babe\c{s}-Bolyai University, Cluj-Napoca, Romania}
\email{calu@math.ubbcluj.ro}

\begin{abstract}
We characterize the unit stable range one $2\times 2$ and $3\times 3$
matrices over commutative rings. In particular, we characterize the $2\times
2$ matrices which satisfy the Goodearl-Menal condition. For $2\times 2$
integral matrices we show that the stable range one and the unit stable
range one properties are equivalent, and, that the only matrix which
satisfies the Goodearl-Menal condition is the zero matrix.
\end{abstract}

\maketitle

\section{Introduction}

The unit stable range one for elements in a unital ring was introduced in 
\cite{goo} and further studied in \cite{chen}.

\textbf{Definition}. An element $a$ in a ring $R$ is said to have \textsl{%
(left) stable range one} (\textsl{sr1}, for short) if whenever $Ra+Rb=R$
(for any $b\in R$) there exists $r\in R$ such that $a+rb$ is a unit. If $r$
can be chosen to be a unit, we say that $a$ has \textsl{(left) unit stable
range one} (\textsl{usr1}, for short). Right (unit) stable range one is
defined symmetrically.

Equivalently, $a$ has \emph{left unit sr1} if for every $x,b\in R$ with $%
xa+b=1$, there is a unit $u\in R$, called \textsl{unitizer} (as in \cite%
{ca-po}), such that $a+ub$ is a unit. Equivalently, for every $x\in R$,
there is a unit $u$ such that $a+u(xa-1)$ is a unit.

By left multiplication with $u^{-1}$ (and change of notation), notice that $%
a $ has \emph{left unit sr1} if and only if for every $x\in R$, there is a
unit $u$ such that 
\begin{equation*}
(u+x)a-1
\end{equation*}%
is a unit.

Recall (well-known as the "\textsl{Jacobson's Lemma}") that for any unital
ring $R$ and elements $\alpha ,\beta \in R$, $1+\alpha \beta $ is a unit if
and only if $1+\beta \alpha $ is a unit. Using the last equivalent
definition, it follows that \emph{the unit stable range one (for elements)
is a left-right symmetric property}. Therefore, in the sequel we chose to
discuss only about \emph{left unit sr1 elements}, removing the "left"
attribute. We also use $usr(a)=1$ to indicate that $a$ has unit sr1. Notice
that \emph{zero} has trivially unit sr1.

Also recall the following

\textbf{Definition}. A ring $R$ is said to satisfy \textsl{the
Goodearl-Menal condition} (GM, for short; see \cite{goo}) provided that for
any $x,y\in R$ there exists a unit $u$ such that both $x-u$, $y-u^{-1}$ are
units.

In this paper, we specialize this to elements of rings, as follows: an
element $a\in R$ \textsl{satisfies the GM condition\ }if for every $x\in R$,
there exists a unit $u$ such that both $x-u$, $a-u^{-1}$ are units. Notice
that $a-u^{-1}$ is a unit if and only if $ua-1$ is a unit, which is a
special case of the unit sr1 definition, for $x=0$.

In Section 2, we give simple properties of elements (and rings) which share
these two properties. In section 3, a characterization for unit sr1 $2\times
2$ and $3\times 3$ matrices over any commutative ring is given, and we show
that unit stable range one and stable range one are equivalent properties
for $2\times 2$ integral matrices. Some special cases, including
idempotents, nilpotents and matrices with zero second row are also discussed.

The last section is dedicated to the Goodearl-Menal condition. A
characterization of the $2\times 2$ matrices over commutative rings which
satisfy the GM condition is given, and it is shown that the only integral $%
2\times 2$ matrix which satisfies the GM condition, is the zero matrix.

All rings we consider are (associative and) unital. For any unital ring $R$, 
$U(R)$ denotes the set of all the units, $N(R)$ denotes the set of all the
nilpotents, $J(R)$ denotes the Jacobson radical of $R$ and for any positive
integer $n\geq 2$, $\mathbb{M}_{n}(R)$ denotes the ring of all the $n\times
n $ matrices over $R$. By $E_{ij}$ we denote a square matrix having all
entries zero, excepting the $(i,j)$ entry, which is $1$.

Whenever it is more convenient, we will use the widely accepted shorthand
\textquotedblleft iff\textquotedblright\ for \textquotedblleft if and only
if\textquotedblright\ in the text.

\section{Prerequisites}

In this section, $R$ denotes an arbitrary (associative and unital) ring. The
following result can be adapted from \cite{goo}.

\begin{lemma}
If $a$ satisfies the GM condition then $a$ has unit sr1.
\end{lemma}

\begin{proof}
If for every $x\in R$ there is a unit $u\in U(R)$ such that both $x-u$, $%
a-u^{-1}$ are units, then $(-(x-u)+x)a-1=ua-1=u(a-u^{-1})\in U(R)$, as
desired.
\end{proof}

Further, recall from \cite{vam} that an element $a$ in a ring $R$ is called 
\textsl{2-good} provided it is a sum of two units. A ring $R$ is \textsl{%
2-good} if so are all its elements.

Half of the condition GM is satisfied by any unit sr1 element.

\begin{lemma}
\label{doi-g}(i) If $usr(a)=1$, there exists a unit $u$ such that $%
a-u^{-1}\in U(R)$.

(ii) If $usr(a)=1$ then $a$ is 2-good.
\end{lemma}

\begin{proof}
(i) Since for every $x\in R$ there is a unit $u\in U(R)$ such that $%
(u+x)a-1\in U(R)$, we just take $x=0$. Hence $ua-1\in U(R)$ and so $%
a-u^{-1}=u^{-1}(ua-1)\in U(R)$.

(ii) Follows from (i), since $a-u^{-1}\in U(R)$ for a unit $u$ iff $a\in
U(R)+U(R)$.
\end{proof}

Therefore we have the following implications%
\begin{equation*}
a\mbox{ satisfies }GM\Rightarrow usr(a)=1\Rightarrow a\in U(R)+U(R).
\end{equation*}%
Examples in $\mathbb{M}_{2}(\mathbb{Z})$ show that these implications are
irreversible: having zero determinant, $E_{11}$ has sr1 (see \cite{ca-po})
but does not satisfy the GM condition (see last section, final step of the
proof of Theorem \ref{zero}), and, $2I_{2}=I_{2}+I_{2}$ is clearly 2-good,
but it has not even sr1 (see \cite{ca-po}: the matrix $nI_{2}$ has sr1 iff $%
n\in \{-1,0,1\}$).

\bigskip

Selecting $r=0$ in the definition of sr1 elements (see Introduction, first
definition) shows that \emph{units have sr1}. However, since $0$ is not a
unit, we cannot use it in order to prove that \emph{units have unit sr1}
(which actually fails: the units $\pm 1$ have not unit sr1 in $\mathbb{Z}$).
Moreover

\begin{lemma}
\label{1}In a ring $R$, $usr(1)=1$ iff $R$ is 2-good.
\end{lemma}

\begin{proof}
According to the equivalent definition, $usr(1)=1$ iff for every $x\in R$,
there is a unit $u$ such that $u+x-1\in U(R)$. Hence $x-1\in U(R)+U(R)$ and
so $R=U(R)+U(R)$. Conversely, for every $x\in R$ there are units $u,v\in
U(R) $ such that $x-1=u+v$. Hence $-u+x-1=v\in U(R)$ and so $usr(1)=1$.
\end{proof}

Since $\mathbb{Z}$ is not 2-good, we infer that $0$ is the only unit sr1
element of $\mathbb{Z}$. This way, $\mathbb{Z}$ is an example of ring whose 
\emph{units have not unit sr1}.

\bigskip

Next, a useful

\begin{lemma}
\label{equiv}(i) If $a$ has unit sr1 and $v\in U(R)$ then $va$ has unit sr1.

(ii) If $a$ has unit sr1, so is $-a$.

(iii) Unit sr1 elements are invariant to conjugations.

(iv) If $a$ has unit sr1 and $v\in U(R)$ then $av$ has unit sr1.

(v) Unit sr1 elements are invariant to equivalences.
\end{lemma}

\begin{proof}
(i) Suppose $x(va)+b=1$. There is $u$ such that $a+ub\in U(R)$. By left
multiplication with $v$ we get $va+vub\in U(R)$, as desired.

(ii) Just take $v=-1$ in (i).

(iii) For every $x$ there is a $u$ such that $a+u(xa-1)\in U(R)$. Then $%
v^{-1}[a+u(xa-1)]v\in U(R)$ but we can write this as $%
v^{-1}av+v^{-1}uv[(v^{-1}xv)(v^{-1}av)-1]$, as desired.

(iv) If $a$ has unit sr1 and $v\in U(R)$ then $v^{-1}av$ has unit sr1, by
(iii). Then by (i), $v(v^{-1}av)=av$ has unit sr1.

(v) Follows from (i) and (iv).
\end{proof}

\section{Unit stable range 1 matrices}

We start with an example of unit sr1 (square) matrix over any ring.

\begin{proposition}
\label{re}Let $R$ be any ring. For any positive integer $n$ and any $r\in R$%
, the matrices $rE_{ij}$ have unit sr1 in $\mathbb{M}_{n}(R)$.
\end{proposition}

\begin{proof}
In \cite{ca-po}, it was proved that the matrices $rE_{ij}$ have sr1, using
the invariance to equivalences, and for any $X\in \mathbb{M}_{n}(R)$, the
unitizer 
\begin{equation*}
Y=\left[ 
\begin{array}{ccccc}
0 & 0 & \cdots & 0 & 1 \\ 
0 & 0 & \cdots & 1 & 0 \\ 
\vdots & \vdots & \ddots & \vdots & \vdots \\ 
0 & 1 & \cdots & 0 & 0 \\ 
1 & 0 & \cdots & 0 & (-1)^{n}a_{1}%
\end{array}%
\right]
\end{equation*}%
where $\mathrm{col}_{1}(X)=\left[ 
\begin{array}{c}
a_{1} \\ 
\vdots \\ 
a_{n}%
\end{array}%
\right] $. Since by Lemma \ref{equiv}, unit sr1 is also invariant to
equivalences and this unitizer is a unit, the statement follows.
\end{proof}

\textbf{Examples}. 1) \emph{Over any ring}, $rE_{12}$ \emph{has usr1}.

For every $X=\left[ 
\begin{array}{cc}
a & b \\ 
c & d%
\end{array}%
\right] $, we can also take the unitizer $U=\left[ 
\begin{array}{cc}
1 & 0 \\ 
-c & 1%
\end{array}%
\right] $. Then $(U+X)rE_{12}-I_{2}=\left[ 
\begin{array}{cc}
-1 & (a+1)r \\ 
0 & -1%
\end{array}%
\right] $ is a unit.

2) \emph{Over any ring}, $rE_{11}$ \emph{has usr1}.

For every $X=\left[ 
\begin{array}{cc}
a & b \\ 
c & d%
\end{array}%
\right] $, we can also take the unitizer $U=\left[ 
\begin{array}{cc}
-a & 1 \\ 
-1 & 0%
\end{array}%
\right] $. Then $(U+X)rE_{11}-I_{2}=\left[ 
\begin{array}{cc}
-1 & 0 \\ 
r(c-1) & -1%
\end{array}%
\right] $ is a unit.

Recall that a ring is called \textsl{GCD} if every pair of elements has a $%
\gcd $ (greatest common divisor).

\begin{corollary}
Over any GCD ring $R\,$, nilpotents and idempotents $\neq I_{2}$ have unit
sr1 in $\mathbb{M}_{2}(R)$.
\end{corollary}

\begin{proof}
This follows, since over any GCD ring, every $2\times 2$ nilpotent is
similar to $rE_{12}$ for some $r\in R$, and every nontrivial idempotent is
similar to $E_{11}$.
\end{proof}

In trying to generalize this result for idempotents, recall that following
Steger \cite{steg}, we say that a ring $R$ is an \textsl{ID} ring if every
idempotent matrix over $R$ is similar to a diagonal one. Thus, by a result
of Song and Guo \cite{son}, if every idempotent matrix over $R$ is
equivalent to a diagonal matrix, then $R$ is an ID ring. Examples of ID
rings include: division rings, local rings, projective-free rings, principal
ideal domains, elementary divisor rings, unit-regular rings and serial rings.

Therefore, over an ID ring, the determination of the idempotent matrices
which have unit sr1, reduces to diagonal matrices.

\begin{corollary}
Over any elementary divisor ring $R$, idempotents of $\mathbb{M}_{n}(R)$
have unit sr1.
\end{corollary}

\begin{proof}
We have to check all block matrices $E=\left[ 
\begin{array}{cc}
I_{m} & \mathbf{0} \\ 
\mathbf{0} & 0_{n-m}%
\end{array}%
\right] $. The statement reduces to the fact that $E_{11}$ has unit sr1 in $%
\mathbb{M}_{n}(R)$ for any $n\geq 2$. Indeed, $E=E_{11}+...+E_{mm}$ $=\left[ 
\begin{array}{cc}
I_{m-1} & 0 \\ 
0 & E_{11}%
\end{array}%
\right] $, where $E_{11}\in \mathbb{M}_{n-m+1}(R)$ has unit sr1 by
Proposition \ref{re} and $I_{m-1}\in \mathbb{M}_{m-1}(R)$ has unit sr1 by
Lemma \ref{1} and the well-known result (see \cite{vam}): any proper matrix
ring over an elementary divisor ring is 2-good. So $E=\left[ 
\begin{array}{cc}
I_{m-1} & 0 \\ 
0 & E_{11}%
\end{array}%
\right] $ has unit sr1 in $\mathbb{M}_{n}(R)$.
\end{proof}

\bigskip

\textbf{Remark}. In \cite{chen} Theorem 2.2.4, it is proved that for any
ring $R$, and positive integers $m$, $n$, if $\mathbb{M}_{m}(R)$ and $%
\mathbb{M}_{n}(R)$ have unit 1-stable range, then so does $\mathbb{M}%
_{m+n}(R)$.

\bigskip

As our first main result, we give a characterization of unit stable range
one for $2\times 2$ and $3\times 3$ matrices over any commutative ring.

\begin{theorem}
\label{cha}(i) Let $R$ be a commutative ring and $A\in \mathbb{M}_{2}(R)$.
Then $A$ has unit stable range 1 iff for any $X\in \mathbb{M}_{2}(R)$ there
exists a unit $U\in \mathbb{M}_{2}(R)$ such that%
\begin{equation*}
\det ((U+X)A)-\mathrm{Tr}((U+X)A)+1
\end{equation*}%
is a unit of $R$.

(ii) Let $R$ be a commutative ring and $A\in \mathbb{M}_{3}(R)$. Then $A$
has unit stable range 1 iff for any $X\in \mathbb{M}_{3}(R)$ there exists a
unit $U\in \mathbb{M}_{3}(R)$ such that%
\begin{equation*}
\det ((U+X)A)-\mathrm{Tr}(adj(U+X)A)+\mathrm{Tr}((U+X)A)-1
\end{equation*}%
is a unit of $R$, where, for any matrix $B\in \mathbb{M}_{3}(R)$, $adj(B)$
denotes the classical adjoint (also called adjugate) of $B$.
\end{theorem}

\begin{proof}
(i) Using the equivalent definition given in the Introduction, $A$ has unit
stable range 1 iff for any $X\in \mathbb{M}_{2}(R)$ there is a unit $U\in 
\mathbb{M}_{2}(R)$ such that $(U+X)A-I_{2}$ is a unit. Since the base ring
is supposed to be commutative, $(U+X)A-I_{2}$ is invertible iff $\det
((U+X)A-I_{2}))$ is a unit of $R$. Since for any $2\times 2$ matrix $C$, $%
\det (C-I_{2})=\det (C)-\mathrm{Tr}(C)+1$, the statement follows.

(ii) The proof is analogous, relying on the formula $\det (C-I_{3})=\det (C)-%
\mathrm{Tr}(adj(C))+\mathrm{Tr}(C)-1$, where here $C$ is any $3\times 3$
matrix.
\end{proof}

\bigskip

As this was done in \cite{ca-po} for sr1, $2\times 2$ matrices, we obtain
alternative proofs for

\begin{corollary}
Let $R$ be a commutative ring and $A\in \mathbb{M}_{2}(R)$ or $A\in \mathbb{M%
}_{3}(R)$. Then $A$ has left unit stable range 1 iff $A$ has right unit
stable range 1.
\end{corollary}

\begin{proof}
Using the properties of determinants, the properties of the trace and the
commutativity of the base ring, it is readily seen that for $2\times 2$
matrices, $\det (U+X)A-\mathrm{Tr}(U+X)A+1=\det A(U+X)-\mathrm{Tr}A(U+X)+1$.
For $3\times 3$ matrices $C$, $D$ we just recall the known formulas 
\begin{equation*}
adj(C)=\dfrac{1}{2}(\mathrm{Tr}^{2}(C)-\mathrm{Tr}(C^{2}))-C\mathrm{Tr}%
(C)+C^{2}
\end{equation*}%
and $\mathrm{Tr}((CD)^{2})=\mathrm{Tr}((DC)^{2})$.

Hence $\mathrm{Tr}(adj(CD))=\dfrac{1}{2}(\mathrm{Tr}^{2}(CD)-\mathrm{Tr}%
((CD)^{2}))=\dfrac{1}{2}(\mathrm{Tr}^{2}(DC)-\mathrm{Tr}((DC)^{2})=\mathrm{Tr%
}(adj(DC))$. All this shows that $\det ((U+X)A-I_{2}))=\det (A(U+X)-I_{2}))$%
, as desired.
\end{proof}

\bigskip

In the sequel, we use the notation $\mathrm{diag}(r,s):=\left[ 
\begin{array}{cc}
r & 0 \\ 
0 & s%
\end{array}%
\right] $. Next, another useful

\begin{lemma}
\label{trans}Over any ring $R$, the following statements hold.

(i) An $n\times n$ matrix $A$\ has unit sr1 iff its transpose $A^{T}$\ has
unit sr1.

(ii) $\mathrm{diag}(r,s)$ has unit sr1 iff $\mathrm{diag}(s,r)$ has unit sr1.

(iii) $\mathrm{diag}(r,s)$ has unit sr1 iff $\mathrm{diag}(r,-s)$ has unit
sr1.
\end{lemma}

\begin{proof}
(i) Indeed, using the left-right symmetry of the unit sr1 property, if $%
(U+X)A-I_{2}$ is a unit, so is its transpose $A^{T}(U^{T}+X^{T})-I_{2}$.

(ii) Follows from Lemma \ref{equiv},(iii), by conjugation with the
involution $E_{12}+E_{21}$.

(iii) Follows from Lemma \ref{equiv} (v), since $\mathrm{diag}(r,-s)$ is
equivalent to $\mathrm{diag}(r,s)$.
\end{proof}

\bigskip

Recall that a commutative unital ring $R$ is an \textsl{elementary divisor}
ring provided every matrix over $R$ is equivalent to a diagonal matrix.
Elementary divisor rings include PIDs, left PIDs which are B\'{e}zout (in
particular division rings), valuation rings and the ring of entire functions.

The definition above, given by M. Henriksen (see \cite{hen}), is more
general than the one given by I. Kaplansky in \cite{kap} and is nowadays in
use.

We have already mentioned that \emph{any proper matrix ring over an
elementary divisor ring is 2-good}. Over an Euclidean domain, proper $%
n\times n$ matrix rings are even \textsl{strongly 2-good} (i.e., a sum of
two invertible matrices that are elementary matrices, permutation matrices
and $-I_{n}$). Hence, in particular, $\mathbb{M}_{2}(\mathbb{Z})$ is
(strongly) 2-good.

\bigskip

A surprising and our second main result is the following

\begin{theorem}
\label{inte}Let $A$ be an integral $2\times 2$ matrix. The following
conditions are equivalent

(i) $A$ has unit sr1,

(ii) $A$ has sr1,

(iii) $\det (A)\in \{-1,0,1\}$.
\end{theorem}

\begin{proof}
(i) $\Rightarrow $ (ii) is obvious and (ii) $\Leftrightarrow $ (iii) was
proved in \cite{ca-po}. Since both (iii) and (i) are invariant to
equivalences (over $\mathbb{Z}$), and $\mathbb{Z}$ is an elementary divisor
ring, for (iii) $\Rightarrow $ (i), it only remains to check that diagonal $%
2\times 2$ matrices with $\det (A)\in \{-1,0,1\}$ have unit sr1. But these
are $0_{2}$, $\pm I_{2}$, $\pm \left[ 
\begin{array}{cc}
1 & 0 \\ 
0 & -1%
\end{array}%
\right] $ and matrices $nE_{11}$ or $mE_{22}$ with nonzero integers $n$, $m$%
. Since $0$ has unit sr1 in any ring, using Lemma \ref{equiv}, (ii) and
Lemma \ref{trans}, and using Proposition \ref{re} for matrices with three
zero entries, it only remains to show that $I_{2}$ has unit sr1. But this
follows from Lemma \ref{1}, since $\mathbb{M}_{2}(\mathbb{Z})$ is (even
strongly) 2-good.
\end{proof}

Since these properties were proved in \cite{ca-po} for sr1 matrices, it
follows at once

\begin{corollary}
(i) In general, unit stable range 1 elements do not have the "complementary
property" (i.e., $1-a$ has unit sr1 whenever $a$ has it).

(ii) In $\mathbb{M}_{2}(\mathbb{Z})$, $AB$ has unit stable range 1 iff so
has $BA$.

(iii) Jacobson's Lemma holds for unit stable range 1 matrices in $\mathbb{M}%
_{2}(\mathbb{Z})$.
\end{corollary}

\textbf{Remarks}. 1) We infer from the above theorem that if an integral $%
2\times 2$ has a unitizer (i.e., has sr1), then it has also a unitizer which
is a unit (i.e., has unit sr1).

2) In view of \cite{ca-po}, the previous proof reduces to check that $I_{2}$
has unit sr1. Of course, we would prefer \textbf{a direct proof} for this.

According to the equivalent definition given in the Introduction, for every
integral $2\times 2$ matrix $X=\left[ 
\begin{array}{cc}
a & b \\ 
c & d%
\end{array}%
\right] $, we should indicate a unit unitizer $U=\left[ 
\begin{array}{cc}
x & y \\ 
z & t%
\end{array}%
\right] $ such that $U+X-I_{2}=V$ is a unit of $\mathbb{M}_{2}(\mathbb{Z})$.
As already mentioned in Lemma \ref{1}, this is done in two steps.

\textbf{Step 1}. We diagonalize $X-I_{2}$, by elementary operations, that
is, find $d\,_{1}$, $d_{2}$ such that $X-I_{2}$ is equivalent to $\mathrm{%
diag}(d_{1},d_{2})$. Over $\mathbb{Z}$ this can be done, using the Euclid's
algorithm (see e.g., \cite{kap} or \cite{zab}). We just sketch this for
reader's convenience.

Using elementary row and column operations, we can equivalently replace $%
X-I_{2}$ by a matrix $M$, which we also denote by $\left[ 
\begin{array}{cc}
a & b \\ 
c & d%
\end{array}%
\right] $, such that $a$ is the least entry in absolute value, among all
matrices that $X-I_{2}$ can be reduced to. Next, if $b=aq_{1}+r_{1}$, $%
a=r_{1}q_{2}+r_{2}$ it can be shown that $r_{2}=0$ (by the minimality in
absolute value of $a$). Hence we can assume $b=aq+r$, $a=rs$ for some
integers $q$, $r$, $s$. Then $\left[ 
\begin{array}{cc}
a & b \\ 
c & d%
\end{array}%
\right] \left[ 
\begin{array}{cc}
1 & -q \\ 
0 & 1%
\end{array}%
\right] \left[ 
\begin{array}{cc}
1 & 0 \\ 
-s & 1%
\end{array}%
\right] \left[ 
\begin{array}{cc}
0 & 1 \\ 
1 & 0%
\end{array}%
\right] =\left[ 
\begin{array}{cc}
r & 0 \\ 
\ast & \ast%
\end{array}%
\right] $, reduces our matrix to (lower) triangular form. In a similar way,
this matrix is reduced to a matrix of form $\left[ 
\begin{array}{cc}
\ast & \ast \\ 
1 & 0%
\end{array}%
\right] $, which is easily diagonalized.

\textbf{Step 2}. Since $\mathrm{diag}(d_{1},d_{2})=\left[ 
\begin{array}{cc}
d_{1} & 1 \\ 
1 & 0%
\end{array}%
\right] +\left[ 
\begin{array}{cc}
0 & -1 \\ 
-1 & d_{2}%
\end{array}%
\right] =-U+V$ is 2-good, these units suit for our above purpose.

Clearly, all this can be done for $2\times 2$ matrices over any Euclidean
domain.

Therefore, it would be difficult to hope for a formula which gives in
general the unit unitizer $U$, given $X$.

\bigskip

\textbf{Examples}. 1) Matrices $M_{uv}=\left[ 
\begin{array}{cc}
1 & u \\ 
v & uv%
\end{array}%
\right] $ have unit sr1 over any commutative ring.

Indeed, having zero determinant, the characterization yields $\det
(U)(1-a-vb-uc-uvd)-t+uz+vy-uvx=\pm 1$ for which the unit unitizer

$U=\left[ 
\begin{array}{cc}
x & y \\ 
z & t%
\end{array}%
\right] =\left[ 
\begin{array}{cc}
0 & 1 \\ 
-1 & 1-a-vb-uc-uvd%
\end{array}%
\right] $ gives $+1$.

2) Simple examples show that $\mathbb{M}_{2}(\mathbb{Z})$ is \emph{not}
closed under addition of unit sr1 matrices. Indeed, both $E_{11}$, $I_{2}$
have unit sr1 but the (diagonal) sum has not.

\bigskip

Our third main result is the following

\begin{theorem}
\label{nm}Let $R$ be a commutative ring. All matrices in $\mathbb{M}_{2}(R)$
with (at least) one zero row or zero column have unit sr1 in any of the
following cases

(i) one entry divides the other;

(ii) $R$ is an Euclidean domain;

(iii) for every $a,c\in R$ there are $q,\alpha ,\beta \in R$ such that $%
(a+qs)\alpha +(-c+qr)\beta =1$ (e.g., $R$ is a B\'{e}zout domain).
\end{theorem}

\begin{proof}
By Lemma \ref{trans} (i), it suffices to prove the claim for matrices with
zero second row.

(i) Let $A=\left[ 
\begin{array}{cc}
r & s \\ 
0 & 0%
\end{array}%
\right] $ with $r,s\in R$ and suppose $s$ divides $r$. Since $\det (A)=0$,
replacement in the characterization theorem gives $-\mathrm{Tr}((U+X)A)+1\in
U(R)$ or $1-r(a+x)-s(c+z)\in U(R)$. If $1-r(a+x)-s(c+z)=1$ we have to solve $%
rx+sz=-ra-sc$ with $xt-yz=1$. We can eliminate $z$, multiplying by $s$: then 
$sxt+y(rx+ra+sc)=s$ which has solution $x=t=1$, $y=0$ and $z=-c-\dfrac{r}{s}%
(a+1)$. Hence $U=\left[ 
\begin{array}{cc}
1 & 0 \\ 
-c-\dfrac{r}{s}(a+1) & 1%
\end{array}%
\right] $ is a suitable unit unitizer. If $r$ divides $s$, similarly a
suitable unit unitizer is $U=\left[ 
\begin{array}{cc}
-a+\dfrac{s}{r}(1-c) & 1 \\ 
-1 & 0%
\end{array}%
\right] $.

(ii) Write $r=sq_{1}+r_{1}$, $s=r_{1}q_{2}+r_{2}$, ..., $%
r_{n-2}=r_{n-1}q_{n}+r_{n}\,$, $r_{n-1}=r_{n}q_{n+1}+0$, for the Euclidean
algorithm where $r_{n}=\gcd (r;s)$. Next, notice that $\left[ 
\begin{array}{cc}
r & s \\ 
0 & 0%
\end{array}%
\right] \left[ 
\begin{array}{cc}
0 & 1 \\ 
1 & -q_{1}%
\end{array}%
\right] =\left[ 
\begin{array}{cc}
s & r_{1} \\ 
0 & 0%
\end{array}%
\right] $, $\left[ 
\begin{array}{cc}
s & r_{1} \\ 
0 & 0%
\end{array}%
\right] \left[ 
\begin{array}{cc}
0 & 1 \\ 
1 & -q_{2}%
\end{array}%
\right] =\left[ 
\begin{array}{cc}
r_{1} & r_{2} \\ 
0 & 0%
\end{array}%
\right] $, ...,

$\left[ 
\begin{array}{cc}
r_{n-2} & r_{n-1} \\ 
0 & 0%
\end{array}%
\right] \left[ 
\begin{array}{cc}
0 & 1 \\ 
1 & -q_{n}%
\end{array}%
\right] =\left[ 
\begin{array}{cc}
r_{n-1} & r_{n} \\ 
0 & 0%
\end{array}%
\right] $, whence $\left[ 
\begin{array}{cc}
r & s \\ 
0 & 0%
\end{array}%
\right] $ is equivalent to $\left[ 
\begin{array}{cc}
r_{n-1} & r_{n} \\ 
0 & 0%
\end{array}%
\right] $. Since $r_{n}$ divides $r_{n-1}$, (i) applies.

(iii) For $X=\left[ 
\begin{array}{cc}
a & b \\ 
c & d%
\end{array}%
\right] $, assume $(a+qs)\alpha +(-c+qr)\beta =1$, for some $q\in R$ and
take the (unit) unitizer $U=\left[ 
\begin{array}{cc}
-a-qs & \beta \\ 
-c+qr & -\alpha%
\end{array}%
\right] $. Then $(U+X)\left[ 
\begin{array}{cc}
r & s \\ 
0 & 0%
\end{array}%
\right] -I_{2}=\left[ 
\begin{array}{cc}
-qrs-1 & -qs^{2} \\ 
qr^{2} & qrs-1%
\end{array}%
\right] $ is a unit of determinant $=1$.
\end{proof}

\section{The Goodearl Menal condition}

Recall that $a\in R$ satisfies the GM condition (rings with the GM condition
were also called "\emph{rings with many units}", for obvious reasons) if for
every $x\in R$ there is $u\in U(R)$ such that both $x-u,a-u^{-1}\in U(R)$.

As already mentioned, in any ring $R$, $a-u^{-1}\in U(R)$ is equivalent to $%
ua-1\in U(R)$.

\textbf{Remark}. For every $x\in R$ there is unit $u$ such that $x-u\in U(R)$%
, is equivalent to $R$ being 2-good. If in addition, $U(R)\cdot
N(R)\subseteq N(R)$ (e.g., reduced or commutative rings), then all
nilpotents satisfy GM.

\bigskip

The following simple result will be useful

\begin{lemma}
\label{u}(i) The GM condition is invariant to equivalences.

(ii) If a unit $a\in U(R)$ satisfies GM then $1$ satisfies GM.
\end{lemma}

\begin{proof}
(i) Let $v\in U(R)$ and suppose $a\in R$ satisfies GM. We show that $va$
also satisfies GM. For an arbitrary $y\in R$ consider $x=yv$. By hypothesis,
there exists $u\in U(R)$ such that $x-u,a-u^{-1}\in U(R)$ and so $yv-u\in
U(R)$ and $a-u^{-1}\in U(R)$. By right multiplication with $v^{-1}$ and left
multiplication with $v$, respectively, we obtain $y-uv^{-1}\in U(R)$ and $%
va-vu^{-1}=va-(uv^{-1})^{-1}\in U(R)$, as desired.

A symmetric proof shows that also $av$ satisfies GM.

(ii) Follows from (i) by left (or right) multiplication with $a^{-1}$.
\end{proof}

\textbf{Remarks}. (i) Notice that $a$ satisfies GM iff $-a$ satisfies GM (by
multiplication with the unit $-1$).

(ii) The GM condition is also invariant to conjugations.

\bigskip

We first prove a characterization of the $2\times 2$ matrices over any
commutative ring which have the GM condition.

\begin{theorem}
A $2\times 2$ matrix $A$ over a commutative ring satisfies the GM condition
iff for every $X$ there is a unit $U$ such that $\det (U)$, $\det (X)+\det
(U)-\mathrm{Tr}(adj(X)U)$ and $\det (U)\det (A)-\mathrm{Tr}(UA)+1$ are units
of $R$.
\end{theorem}

\begin{proof}
Analogous with the proof of Theorem \ref{cha}.
\end{proof}

For the special case of integral matrices we obtain

\begin{corollary}
\label{GM1}An integral matrix $A$ satisfies the GM condition iff for every $%
X $ there is a unit $U$ such that

(a) $\mathrm{Tr}(UA)\in \{-1,1,3\}$, $\det (X)-\mathrm{Tr}(adj(X)U)\in
\{-2,0,2\}$ and $A$ is a unit, or else

(b) $\mathrm{Tr}(UA)\in \{0,2\}$, $\det (X)-\mathrm{Tr}(adj(X)U)\in
\{-2,0,2\}$ and $\det (A)=0$.
\end{corollary}

\begin{proof}
For $X=\left[ 
\begin{array}{cc}
a & b \\ 
c & d%
\end{array}%
\right] $, $U=\left[ 
\begin{array}{cc}
x & y \\ 
z & t%
\end{array}%
\right] $ and any $2\times 2$ matrix $A$, both $X-U$ and $UA-I_{2}$ are
units, iff%
\begin{eqnarray*}
\det (U) &\in &\{\pm 1\}, \\
\det (X)+\det (U)-\mathrm{Tr}(adj(X)U) &\in &\{\pm 1\}, \\
\det (U)\det (A)-\mathrm{Tr}(UA)+1 &\in &\{\pm 1\}.
\end{eqnarray*}

\textbf{Case 1}. Let $\det (U)=1$. Then $\det (X)-\mathrm{Tr}(adj(X)U)\in
\{-2,0\}$ and $\det (A)-\mathrm{Tr}(UA)+1\in \{\pm 1\}$. The first condition
is independent from $A$.

(a) Since $A$ is supposed to be a unit, $1-\mathrm{Tr}(UA)\pm 1\in \{\pm 1\}$%
, that is, $\mathrm{Tr}(UA)\in \{-1,1,3\}$.

(b) If $\det (A)=0$, we get $\mathrm{Tr}(UA)\in \{0,2\}$

\textbf{Case 2}. Let $\det (U)=-1$. Then $\det (X)-\mathrm{Tr}(adj(X)U)\in
\{0,2\}$ and $-\det (A)-\mathrm{Tr}(UA)+1\in \{\pm 1\}$.

(a) Since $A$ is supposed to be a unit, $1-\mathrm{Tr}(UA)\mp 1\in \{\pm 1\}$%
, that is, the same $\mathrm{Tr}(UA)\in \{-1,1,3\}$.

(b) If $\det (A)=0$, again we get $\mathrm{Tr}(UA)\in \{0,2\}$.
\end{proof}

It is easy to see that $0$ satisfies the GM condition in a ring $R$ iff $R$
is 2-good. Therefore, $0_{2}$ satisfies GM in $\mathbb{M}_{2}(\mathbb{Z})$.

Next, with the expression mentioned in starting this section, we show that $%
\mathbb{M}_{2}(\mathbb{Z})$ has "few" units. More precisely, our fourth and
last main result follows.

\begin{theorem}
\label{zero}Nonzero matrices of $\mathbb{M}_{2}(\mathbb{Z})$ do not satisfy
GM.
\end{theorem}

\begin{proof}
As noticed in Section 2, elements which satisfy the GM condition, have
(unit) sr1. Hence, with respect to integral $2\times 2$ matrices (see also
Theorem \ref{inte}), these are units or have zero determinant. Therefore we
split the proof in two parts.

We first show that \emph{units do not satisfy GM in} $\mathbb{M}_{2}(\mathbb{%
Z})$. Using Lemma \ref{u}, it suffices to show that $I_{2}$ does not satisfy
GM in $\mathbb{M}_{2}(\mathbb{Z})$.

According to Corollary \ref{GM1}, $I_{2}$ has GM iff for every $X$ there is
a unit $U$ such that $\mathrm{Tr}(U)\in \{-1,1,3\}$ and $\det (X)-\mathrm{Tr}%
(adj(X)U)\in \{-2,0,2\}$.

Here $adj(X)=\left[ 
\begin{array}{cc}
d & -b \\ 
-c & a%
\end{array}%
\right] $, so the second condition becomes $ad-bc-dx+bz+cy-at\in \{-2,0,2\}$.

Hence $x+t\in \{-1,1,3\}$, $xt-yz=1$ and $dx-cy-bz+at=ad-bc+k$ with $k\in
\{-2,0,2\}$, which are Diophantine linear equations (for which the
solvability condition is well-known).

\textbf{Case 1}. $x+t=-1$ or $t=-x-1$ gives 
\begin{equation*}
(d-a)x-cy-bz=a(d+1)-bc+k.
\end{equation*}
If we take $X=\left[ 
\begin{array}{cc}
2 & 5 \\ 
5 & 7%
\end{array}%
\right] $ then $\gcd (d-a;c;b)=5$ and $a(d+1)-bc=16-25=-9$. Since $-9+k$ is
not divisible by $5$, there are no integer solutions.

\textbf{Case 2}. $x+t=1$ or $t=1-x$ gives 
\begin{equation*}
(d-a)x-cy-bz=a(d-1)-bc+k.
\end{equation*}
If we take $X=\left[ 
\begin{array}{cc}
3 & 5 \\ 
5 & 8%
\end{array}%
\right] $ then $\gcd (d-a;c;b)=5$ and $a(d-1)-bc=21-25=-4$. Since $-4+k$ is
not divisible by $5$, there are no integer solutions.

\textbf{Case 3}. $x+t=3$ or $t=3-x$ gives 
\begin{equation*}
(d-a)x-cy-bz=a(d-3)-bc+k.
\end{equation*}
If we take $X=\left[ 
\begin{array}{cc}
4 & 5 \\ 
5 & 9%
\end{array}%
\right] $ then $\gcd (d-a;c;b)=5$ and $a(d-3)-bc=24-25=-1$. Since $-1+k$ is
not divisible by $5$, there are no integer solutions.

Secondly, we show that the only \emph{zero determinant matrix} which
satisfies GM in $\mathbb{M}_{2}(\mathbb{Z})$ is the zero matrix.

Since $\mathbb{Z}$ is an elementary divisor ring, every matrix over $\mathbb{%
Z}$ is equivalent to a diagonal matrix, and using Lemma \ref{u}, it suffices
to prove our claim for diagonal matrices of zero determinant. Excepting the
zero matrix and (if necessary) using Lemma \ref{trans} (i), it suffices to
check this for $nE_{11}$. Since $nE_{11}$ satisfies GM iff $-nE_{11}$
satisfies GM, we can assume $n$ a positive integer.

For multiples $A=nE_{11}$ with $n\geq 2$, take $X=\left[ 
\begin{array}{cc}
0 & 4 \\ 
4 & 3%
\end{array}%
\right] $ and $U=\left[ 
\begin{array}{cc}
x & y \\ 
z & t%
\end{array}%
\right] $. Then $\mathrm{Tr}(UA)=nx\in \{0,2\}$, holds only for $x=0$ if $%
n\geq 3$ and for $x\in \{0,1\}$ if $n=2$. Further, $\det (X)-\mathrm{Tr}%
(adj(X)U)=-16-3x+4y+4z$. Then if $x=0$, since $xt-yz=\det (U)\in \{\pm 1\}$, 
$yz\in \{\pm 1\}$ follows and so $y,z\in \{\pm 1\}$. The condition becomes $%
4(-4+y+z)\in \{-2,0,2\}$, impossible for $y,z\in \{\pm 1\}$. If $x=1$, then $%
-19+4(y+z-t)\notin \{-2,0,2\}$.

Finally, we show that $A=E_{11}$, does not satisfy GM. With an arbitrary $X=%
\left[ 
\begin{array}{cc}
a & b \\ 
c & d%
\end{array}%
\right] $ and $U=\left[ 
\begin{array}{cc}
x & y \\ 
z & t%
\end{array}%
\right] $, the conditions $\det (X)-\mathrm{Tr}(adj(X)U)\in \{-2,0,2\}$ and $%
\mathrm{Tr}(UA)\in \{0,2\}$ become $x\in \{0,2\}$ and $ad-bc-dx+cy+bz-at\in
\{-2,0,2\}$.

(i) For $x=0$ we have $ad-bc+cy+bz-at\in \{-2,0,2\}$ which are Diophantine
equations%
\begin{equation*}
cy+bz-at=-ad+bc+k
\end{equation*}%
where $k\in \{-2,0,2\}$.

If $\delta =\gcd (c;b;a)$ and $\delta \geq 3$, the equations are solvable
only if $k=0$. In this case, $\delta $ divides also $-ad+bc+0$, so this
equation has always solutions. However, we need a solution such that $U$ is
a unit, i.e., $xt-yz\in \{\pm 1\}$.

Take $X=\left[ 
\begin{array}{cc}
15 & 3 \\ 
3 & 0%
\end{array}%
\right] $ and divide $a$, $b$, $c$ by $\delta =3$. Then we obtain the
Diophantine equation $y+z-5t=1$. Since as already mentioned, $y,z\in \{\pm
1\}$, we have $y+z\in \{-2,0,2\}$, so the equation has no solutions.

If $\delta =2$, the equations are solvable for any $k\in \{-2,0,2\}$.
Dividing by $2$, we can assume coprime $a$, $b$, $c$ and $k\in \{-1,0,1\dot{%
\}}$. Again $y,z\in \{\pm 1\}$, whence $2y+5z\in \{-7,-3,3,7\}$. Starting
with $X=\left[ 
\begin{array}{cc}
22 & 10 \\ 
4 & 0%
\end{array}%
\right] $, we get $2y+5z-11t=-11d+10+k$, with $k\in \{-1,0,1\dot{\}}$. The
LHS is congruent (mod $11$) to $3$ or $4$ or $7$ or $8$ but the RHS is
congruent to $0$ or $9$ or $10$. So the equation has no solutions.

If $\delta =1$, the equations are solvable for any $k\in \{-2,0,2\}$. Again $%
y,z\in \{\pm 1\}$ whence $y+2z$ is odd. Starting with $X=\left[ 
\begin{array}{cc}
8 & 2 \\ 
1 & 0%
\end{array}%
\right] $, we get $y+2z-8t=-8d+2+k$. Since LHS is odd and RHS is even, the
equation has no solutions.

(ii) For $x=2$ we have 
\begin{equation*}
cy+bz-at=(2-a)d+bc+k
\end{equation*}%
with $k\in \{-2,0,2\}$. We take $X=\left[ 
\begin{array}{cc}
5 & 5 \\ 
5 & 2%
\end{array}%
\right] $ and so $\gcd (c;b;a)=5$ but $(2-a)d+bc=-6+25=19$ and $19+k$ is not
divisible by $5$. This completes the proof.
\end{proof}

\section{Open questions}

Give examples of:

\textbf{1)} an element of $J(R)-N(R)$ which has not unit sr1.

Hint: \emph{nilpotents in the Jacobson radical have unit sr1}. Indeed, since 
$a\in J(R)$ iff $1-xa$ is left invertible for any $x\in R$, for the (unit)
sr1 property, we can choose the (unit) unitizer $y=-(1-a)(1-xa)^{-1}$, where 
$1-a\in U(R)$ if $a\in N(R)$.

\bigskip

\textbf{2)} two unit sr1 elements whose product has not unit sr1.

Hint: according to Theorem \ref{inte}, and T.Y. Lam's proof for the
multiplicative closure of sr1 elements (see \cite{L1}), such an example
cannot be given in $\mathbb{M}_{2}(\mathbb{Z})$.

\bigskip

\textbf{3)} Which idempotents have unit sr1 ? If $R$ is 2-good, have all
idempotents unit sr1 ?

Idempotents are unit-regular and unit-regular elements have sr1. As already
noticed, the idempotent $1$ has not unit sr1 in $\mathbb{Z}$.

For matrix rings these questions were addressed in Section 3.

\bigskip

\begin{acknowledgement}
Thanks are due to Horia F. Pop for computer aid.
\end{acknowledgement}

\end{document}